\documentclass[a4paper,11pt]{article}
\usepackage[sc]{mathpazo}
\usepackage[T1]{fontenc}
\usepackage{amsthm,amsmath, amsfonts,hyperref,a4wide,color}

\newtheorem{theorem}{Theorem}
\newtheorem{lemma}[theorem]{Lemma}
\newtheorem{cor}[theorem]{Corollary}
\newtheorem{proposition}[theorem]{Proposition}

\newtheorem{claim}{Claim}
\newenvironment{claimproof}[1]{\par\noindent\emph{Proof.}\space#1}{\hfill $\Diamond$\newline}
\theoremstyle{definition}

\theoremstyle{remark}

\newcommand*{\Z}{\mathbb{Z}}

\newcommand*{\N}{\mathbb{N}}

\newcommand{\Hyp}[3]{H^{(#1)}_{#2,#3}}

\newcommand*{\qq}{q}
\newcommand*{\rr}{r}

\newcommand{\eps}{\varepsilon}

\renewcommand\labelenumi{(\roman{enumi})}
\renewcommand\theenumi\labelenumi

\title{Discrepancy and large dense monochromatic subsets}
\author{Ross J. Kang\thanks{Department of Mathematics, IMAPP, Radboud University Nijmegen. Email: \texttt{ross.kang@gmail.com}. This work was initiated while this author was supported by a Veni grant from the Netherlands Organisation for Scientific Research (NWO).}\and Viresh Patel\thanks{Korteweg de Vries Institute for Mathematics, University of Amsterdam. Email: \texttt{vpatel@uva.nl}. Supported by NWO through the Gravitation Programme Networks (024.002.003). } \and Guus Regts\thanks{Korteweg de Vries Institute for Mathematics, University of Amsterdam. Email: \texttt{guusregts@gmail.com.} Supported by a personal NWO Veni grant.} }

\begin{document}

\maketitle
\abstract{
Erd\H{o}s and Pach (1983) introduced the natural degree-based generalisations of Ramsey numbers, where instead of seeking large monochromatic cliques in a $2$-edge coloured complete graph, we seek  monochromatic subgraphs of high minimum or average degree. Here we expand the study of these so-called quasi-Ramsey numbers in a few ways, in particular, to multiple colours and to uniform hypergraphs.

Quasi-Ramsey numbers are known to exhibit a certain unique phase transition and we show that this is also the case across the settings we consider. 
Our results depend on a density-biased notion of hypergraph discrepancy optimised over sets of bounded size, which may be of independent interest.
\\
\quad \\
\footnotesize{{{\bf Keywords.} Ramsey theory, quasi-Ramsey numbers, hypergraph discrepancy, probabilistic method}
\\\quad\\
\footnotesize{{\bf AMS subject classifications.} Primary, 05C55; Secondary, 05D10, 05D40

}

\section{Introduction}

Frank Plumpton Ramsey~\cite{Ram30} originally addressed the following question. Fixing $\qq,\rr\ge2$, for any $k$, is there always a finite $n$ such that in any assignment of $\qq$ colours to the $\rr$-element subsets of $[n]=\{1,\dots,n\}$, there is guaranteed to be a $k$-element subset of $[n]$ all of whose $\rr$-element subsets have the same colour? Ramsey's Theorem states that the answer is yes. The search for the smallest values $R^{(\rr)}_\qq(k)$ of $n$ in this question (the Ramsey numbers) is a central part of combinatorial mathematics. This search was begun in seminal papers by Erd\H{o}s and Szekeres~\cite{ErSz35} and Erd\H{o}s~\cite{Erd47} for the case $\qq=\rr=2$ showing that
\begin{align*}
\sqrt{2}^k \le R^{(2)}_2(k) \le 4^k.
\end{align*}
After decades, these remain very near to the best known bounds for this parameter.

When $\qq>2$ or $\rr>2$, our knowledge of the situation is even worse.
If $\qq>2$ and $\rr=2$, then $R^{(2)}_\qq(k)$ is exponential in $k$, but the best known bounds on the constants in the base of the exponential are weaker for larger $\qq$.
More significantly, if $\rr>2$ and $\qq=2$, then $R^{(\rr)}_2(k)$ is known to grow like a tower of exponentials in $k$~\cite{EHR65}, but the height of this tower is unknown and is subject to a \$500 Erd\H{o}s prize.
On the other hand, note that Erd\H{o}s and Hajnal (cf.~\cite{GRS80}) have shown that $\ln \ln R^{(3)}_4(k) = \Theta(k)$ as $k\to\infty$.

For $\qq=\rr=2$, Erd\H{o}s and Pach~\cite{ErPa83} formulated a natural degree-based generalisation of the Ramsey numbers.
Given $c\in [0,1]$, the basic question is as follows: for any $k$, what is the smallest $n:=R_c(k)$ such that for any graph $G=(V,E)$ on $n$ vertices there exists a subset $S\subseteq V$ of size $\ell$ at least $k$ such that either $G[S]$ or its complement $\overline{G}[S]$ has minimum degree at least $c(\ell-1)$?
We may also ask this question with average degree instead of minimum degree and denote the corresponding number $\overline{R}_c(k)$.
Clearly $\overline{R}_c(k) \le R_c(k)$ always.
We refer to $\overline{R}_c(k)$ and $R_c(k)$ as \emph{quasi-Ramsey numbers}.
Of course by taking $c=1$ we recover the classical two-colour Ramsey numbers for graphs.

Erd\H{o}s and Pach~\cite{ErPa83} found that the quasi-Ramsey numbers undergo a dramatic change in growth in $k$ in a narrow window around $c=1/2$: if $c<1/2$ then they have linear growth, while if $c>1/2$ they have singly exponential growth.
They developed a fairly precise understanding of the transition at the point $c=1/2$ --- the present authors together with Pach~\cite{KPPR15} and with Long~\cite{KLPR14+} have recently refined this.

Our purpose in the present paper is to extend the study of quasi-Ramsey numbers to multiple colours and uniform hypergraphs (as was initially considered by Ramsey).

We have been able to show that the precise transition behaviour at the point $c=1/2$ for $\rr=\qq=2$ is present in a similar way when $\rr>2$ or $\qq>2$.
The proofs of our results rely critically on a density-biased notion of hypergraph discrepancy which is optimised only over those vertex subsets up to a certain size. In fact the most difficult part of the paper is devoted to proving a bound for this type of discrepancy, cf. Theorem~\ref{thm:skew disc} below, which we believe to be of independent interest.

\subsection{Multi-colour quasi-Ramsey for graphs}

For the case $\rr=2$, we would first like to study the behaviour of quasi-Ramsey numbers if rather than two colours (namely the graph and its complement) there are $\qq \ge 2$ colours assigned to the edges of $K_n$. Motivated partly by related recent work by Falgas-Ravry, Markstr\"om, and Verstra\"ete~\cite{FMV15+}, we treat an even more general setting where each of the $\qq$ colours has an associated ``degree share''.
Based on Theorem~\ref{thm:skew disc} below, we prove the following in Section~\ref{sec:QR graphs}. 
\begin{theorem} \label{thm:graph multi-colour}
Fix $\qq\ge2$, $(\rho_1,\dots,\rho_\qq)\in (0,1)^\qq$ such that $\sum_{i=1}^\qq\rho_i=1$, and $\nu\ge 0$. Then there exists a constant $C=C(\nu)>0$ such that for each $k$ large enough and any $\qq$-colouring of the edges of the complete graph on at least $Ck\ln k$ vertices, there exists a colour $j$ and a subset $S$ of the vertices of size $\ell\ge k$ such that the subgraph induced by $S$ in colour $j$ has minimum degree at least $\rho_j(\ell-1)+\nu \sqrt{\ell-1}$.
\end{theorem}
By a clever weighted random construction (cf.~also~\cite{KPPR15}), Erd\H{o}s and Pach proved that $R_{1/2}(k)= \Omega(k\ln k/\ln\ln k)$, which means that the quasi-Ramsey number bound  for $c=1/2$ implicit in Theorem~\ref{thm:graph multi-colour} is sharp up to a $\ln\ln k$ factor.

We remark that Theorem~\ref{thm:graph multi-colour} gives progress on a question posed by Falgas-Ravry, Markstr\"om, and Verstra\"ete~\cite{FMV15+}.
Given a graph $G$ on $n$ vertices with edge density $p$, they asked for the largest integer $m = g(G)$ such that $G$ contains an induced subgraph on at least $m$ vertices with minimum degree at least $p(m-1)$ (what they called a {\em full subgraph}) or with maximum degree at most $p(m-1)$ (a {\em co-full subgraph}).
In an earlier version of \cite{FMV15+}, the authors showed that if $p(1-p)\ge 1/n$ then $g(G)=\Omega(n/(\ln n)^2)$ for all graphs $G$, and asked whether this bound could be improved to $\Omega(n/\ln n)$. 
In the latest version of \cite{FMV15+}, they show $g(G)=\Omega(n/(\ln n))$ and no longer require the condition $p(1-p)\ge 1/n$ (see Theorem 4). Here (addressing the question in the earlier version) we obtain (a strengthening and generalisation of) the same result via Theorem~\ref{thm:graph multi-colour} and Corollary~\ref{cor:nu=0}.
Indeed, in the case where the edge density $p$ is fixed, Theorem~\ref{thm:graph multi-colour} is a strengthening since we can guarantee slightly higher degree than required by taking $\qq=2$ and $\rho_1=p$. It is a generalisation in the sense of allowing more colours.
In Section~\ref{sec:QR graphs}, we show that this $\Omega(n/\ln n)$ bound is also valid for non-constant $p$, cf.~Corollary~\ref{cor:nu=0}. 

\subsection{Multi-colour quasi-Ramsey for hypergraphs}

The multicolour quasi-Ramsey investigation above naturally extends also to $\rr$-uniform hypergraphs, where we consider colourings of the hyperedges of the complete $\rr$-uniform hypergraph $K^{(\rr)}_n$ on $n$ vertices.
The \emph{degree} of a vertex is the number of hyperedges incident with the vertex.

As Ramsey numbers for hypergraphs are even less well understood than for graphs, despite a long history, one might expect the hypergraph quasi-Ramsey problem to put up significant resistance.
To the contrary, we have found that the precise threshold in quasi-Ramsey numbers for graphs established in~\cite{KPPR15} is present in an analogous way for hypergraphs.
Based on Theorem~\ref{thm:skew disc} below 
and a standard random construction
we establish the following result in Section~\ref{sec:QR hyper}.

\begin{theorem}\label{thm:sharp hyper}
Let $\rr\ge2$. Fix $\qq \ge 2$ and $(\rho_1,\dots,\rho_\qq)\in (0,1)^\qq$ with $\sum_{i=1}^\qq\rho_i=1$.
\begin{enumerate}
\item\label{itm:sharp hyper,upper}
Let $\nu\ge 0$. Then there exists a constant $C>0$ such that for each $\varepsilon>0$ and $k$ large enough,
for any $\qq$-colouring of the edges of the complete $\rr$-uniform hypergraph on at least $k^{\nu^2C^2(1+\varepsilon)+2\rr/(\rr+1)}$ vertices there exists a colour $j$ and a subset $S$ of the vertices of size $\ell\ge k$ such that the subhypergraph induced by $S$ in colour $j$ has minimum degree at least 
\begin{align*}
\rho_j\binom{\ell-1}{\rr-1}+\nu \sqrt{\ell^{\rr-1}\ln \ell}.
\end{align*}

\item\label{itm:sharp hyper,lower}
There is a constant $C>0$ such that, if $\nu(\cdot)$ is a non-decreasing non-negative function, then for each $k$ large enough there is a 
$\qq$-colouring of the edges of the complete $\rr$-uniform hypergraph on $Ck^{\nu(k)^2+1}$ vertices such that the following holds. For any colour $j$ and any subset $S$ of the vertices of size $\ell\ge k$, the subhypergraph induced by $S$ in colour $j$ has average degree less than
\begin{align*}
\rho_j\binom{\ell-1}{\rr-1}+\nu(\ell) \sqrt{\rr\binom{\ell-1}{\rr-1}\ln \ell}.
\end{align*}

\end{enumerate}
\end{theorem}

We note that if we wish to find induced subgraphs with exactly (rather than at least) $k$ vertices, then the following applies for for $\sum_{i=1}^\qq\rho_i<1$. The proof appears in Section~\ref{sec:QR hyper}.
\begin{proposition}\label{prop:linear hyper}
Let $\rr\ge2$. Fix $\qq \ge 2$ and $(\rho_1,\dots,\rho_\qq)\in [0,1]^\qq$ with $\sum_{i=1}^\qq\rho_i<1$.
Then there exists a constant $C>0$ such that for each $k$ large enough,
for any $\qq$-colouring of the edges of the complete $\rr$-uniform hypergraph on at most $Ck$ vertices there exists a colour $j$ and a subset $S$ of the vertices of size $k$ such that the subhypergraph induced by $S$ in colour $j$ has minimum degree at least $\rho_i\binom{k-1}{r-1}$.\end{proposition}

The situation could be more nuanced if $\sum_{i=1}^\qq\rho_i>1$. It is of course unknown what precisely happens when $\sum_{i=1}^\qq\rho_i=\qq$, that is, the regime of the classical hypergraph Ramsey numbers, but we also do not know much for $1<\sum_{i=1}^\qq\rho_i<\qq$.
We will elaborate on this and state some open questions in Section~\ref{sec:remarks}.

\paragraph{Organisation.}
As mentioned above the proofs of our main results rely on a discrepancy result, which we state and prove in the next section. 
In Section~\ref{sec:QR graphs} we prove Theorem~\ref{thm:graph multi-colour} and in Section~\ref{sec:QR hyper} we prove Theorem~\ref{thm:sharp hyper} and Proposition~\ref{prop:linear hyper}.
We conclude with some remarks and open questions in Section~\ref{sec:remarks}.

\section{Discrepancy over sets of bounded size}
In this section we introduce our main tool, a $p$-discrepancy result for bounded sets in uniform hypergraphs, cf. Theorem~\ref{thm:skew disc} below.

Let $\rr\ge2$ and let $H=(V,E)$ be an $\rr$-uniform hypergraph.
For $p\in [0,1]$ and $S\subseteq V$, the \emph{$p$-discrepancy} of $S$ is defined as
\begin{align*}
D_p(S):=e(S)-p\binom{|S|}{\rr},
\end{align*}
the number of hyperedges in the subhypergraph induced by $S$ less a $p$ proportion of the total possible number of hyperedges on $S$.
For several $\rr$-uniform hypergraphs defined on the same vertex set, we specify $D_{p,H}(S)$.
The \emph{$p$-discrepancy} of $H$ is defined as
\begin{align}
D_p(H) := \max_{S\subseteq V} |D_p(S)|.\label{eq:disc df}
\end{align}
For the classic choice $p=1/2$, we usually refer to this just as discrepancy. 
If $p$ is chosen as $|E|/\binom{|V|}{r}$, the hyperedge density of $H$, then the $p$-discrepancy measures how uniformly the hyperedges are distributed over the vertices.

A well-known result of Erd\H{o}s and Spencer~\cite{ErSp72} states that there exists $C=C(\rr)>0$ such that, provided $n$ is large enough, the discrepancy of any $\rr$-uniform hypergraph $H=(V,E)$ on $n$ vertices satisfies
\begin{align}\label{eq:ES72}
D_{1/2}(H)\ge Cn^{(\rr +1)/2}.
\end{align}
This is sharp up to the choice of the constant $C$.
The same statement for $p$-discrepancy with $p=|E|/\binom{|V|}{\rr}$ was shown by Erd\H{o}s, Goldberg, Pach and Spencer~\cite{EGPS88} for $\rr=2$ and by Bollob\'as and Scott~\cite{BoSc06} for $\rr > 2$ (where the constant $C$ depends on $p$).

It is natural to wonder what happens when the sets over which the maximum is taken in \eqref{eq:disc df} all have a bounded number $t$ of vertices. 
Clearly, one can obtain a constant times $t^{(r+1)/2}$, but in fact one gains a little more.  
Below we prove the following, generalising the results of Erd\H{o}s and Spencer~\cite{ErSp72,ErSp74}, of Erd\H{o}s, Goldberg, Pach and Spencer~\cite{EGPS88} and, partially, of Bollob\'as and Scott~\cite{BoSc06}.
\begin{theorem}\label{thm:skew disc}
Let $\rr\ge2$. There exist constants $C,D>0$ such that for any $p\in (0,1)$ the following holds. For each $n$ large enough and all $(\ln n)/D\le t \le n$,
we have that any $\rr$-uniform hypergraph $H=(V,E)$ on $n$ vertices satisfies 
\begin{align}
\max_{{S\subseteq V, |S|\le t}} |D_p(S)| \ge C\min\{p,1-p\}t^{(\rr+1)/2}\sqrt{\ln(n/t)}.	\label{eq:skew disc}
\end{align}
\end{theorem}

Note that $p$ in Theorem~\ref{thm:skew disc} is not assumed to be the density of the hypergraph. 
We also note that the case $p=1/2$ of Theorem~\ref{thm:skew disc} was proved for $\rr=2$ (i.e.\ graphs) and shown to be tight up to the choice of the constant $C$ by Erd\H{o}s and Spencer \cite[Theorem~7.1]{ErSp74}. A slightly stronger form for the hypergraph case (for $p=1/2$) was announced and its proof left as a ``difficult'' exercise in~\cite[Chapter~7]{ErSp74}.
To the best of our knowledge no proof has been published.
Although Theorem~\ref{thm:skew disc} suffices for our purposes, for $p$ varying as a function of $n$, there is still room for potential improvement in the bound, since a random $\rr$-uniform hypergraph with edge density $p$ supplies an upper bound example with instead the factor $\min\{\sqrt{p},\sqrt{1-p}\}$ in~\eqref{eq:skew disc}.

\subsection{Proof of Theorem~\ref{thm:skew disc}}\label{sec:disc}
Our proof may be viewed as an extension of the proof of Erd\H{o}s and Spencer~\cite{ErSp72} of \eqref{eq:ES72}.
We will first prove several lemmas extending lemmas from~\cite{ErSp72}.
We start with the following adaptation of~\cite[Lemma~2]{ErSp72}. 
\begin{lemma}\label{lem:1-disc}
Fix $c>0$. Then, for all $m$, all $y\geq 2$ such that $\ln y\leq cm/4$, and any choice of real numbers $x_1,\dots,x_m$ satisfying $|x_i|\ge 1$ for at least $cm$ of the $i\in[m]$, we have
\begin{align}
\left| \sum_{i\in V}x_i\right|\ge 4^{-1}\sqrt{cm \ln (y)},	    \label{eq:1-disc}
\end{align}
for at least $(8y)^{-1} 2^{m}$ choices of $V\subseteq [m]$.
\end{lemma}
\begin{proof}
For simplicity, let us assume that $cm\in\N$ and that $x_1,\dots,x_{cm}$ all have absolute value at least $1$.
For $V\subseteq [m]$ set $\phi(V):=\sum_{i\in V} x_i$, $V_1:=V\cap[cm]$ and $V_2:=V\setminus V_1$.
Then $\phi(V)=\phi(V_1)+\phi(V_2)$. 
Set $c_1=16^{-1}cm\ln y$.
Then~\eqref{eq:1-disc} does not hold if and only if $\phi(V_1)\in (-\phi(V_2)- \sqrt{c_1},-\phi(V_2)+ \sqrt{c_1})$.
By a result of Erd\H{o}s~\cite{Erd45} this holds for fixed $V_2$ for at most 
\[
\sum_{r \; : \; |r-\frac{cm}{2}|\le  \sqrt{c_1}} \binom{cm}{r}	\]
choices of $V_1\subseteq [cm]$.
Since $\sqrt{c_1}\le cm/8$ by assumption on $y$, it follows from elementary arguments, cf.~Proposition~7.3.2 in the lecture notes of Matou\v{s}ek and Vondr\'ak~\cite{MaVonotes}, that
\[
2^{cm}-\sum_{r \; : \; |r-\frac{cm}{2}|\le  \sqrt{c_1}} \binom{cm}{r}	\geq \frac{2^{cm+1}}{15}\exp\left(\frac{-16c_1}{cm}\right)=\frac{2^{cm+1}}{15y}>\frac{2^{cm}}{8y}.
\]
In other words, for fixed $V_2$, we have for at least $(8y)^{-1}2^{cm}$ choices of $V_1$ that~\eqref{eq:1-disc} holds.
Now summing over all possible $V_2$ proves the lemma.
\end{proof}

We will next prove a result about $\rr$-partite $\rr$-uniform hypergraphs, or $(\rr,\rr)$-graphs for short.
Recall that an $\rr$-uniform hypergraph $H=(V,E)$ is said to be \emph{$\rr$-partite} if there exists a partition of $V$ into $\rr$ sets $V_1,\dots,V_{\rr}$ such that every hyperedge $e$ intersects  all of the $V_i$ exactly once. In this case we sometimes say $H$ is an $(\rr,\rr)$-graph on $V_1 \cup \cdots \cup V_r$.
For an $\rr$-uniform hypergraph $H=(V,E)$ and pairwise disjoint subsets $S_1,\dots, S_\rr\subseteq V$, define 
$e(S_1,\dots,S_\rr)$ to be the number of hyperedges of $H$ that have exactly one endpoint in $S_i$ for $i\in[\rr]$.
Then define 
\begin{align*}
D_p(S_1,\dots,S_\rr):=e(S_1,\dots,S_\rr)-p\prod_{i=1}^{\rr} |S_i|.
\end{align*}
The next lemma extends~\cite[Lemma~1]{ErSp72}.

\begin{lemma}\label{lem:multi disc}
Let $\rr\ge2$ and $p\in (0,1)$. There exists constants $c_\rr,d_\rr>0$ such that for all $t$ large enough, all $y\geq 2$ such that $\ln y\leq d_{r}t$, and any $(\rr,\rr)$-graph $H$ on $A_1\cup \dots
\cup A_\rr$, with $|A_i|=t$ for each $i$, we have $|D_p(B_1,\dots,B_\rr)|\ge \min\{p,1-p\}c_\rr t^{\rr/2}\sqrt{\ln y}$ for at least $y^{-1}d_\rr2^{t\rr}$ choices of subsets $B_i\subseteq A_i$, $i\in[\rr]$.
\end{lemma}
\begin{proof}
The proof is by induction on $\rr$.
In case $\rr=1$, we have for a $(1,1)$-graph $H=(V,E)$ and $S\subseteq V$ that $D_p(S)=\sum_{i\in S}x_i$, with $x_i=1-p$ if $i\in E$ and $x_i=-p$ if $i\notin E$.
Let $\hat{p}=\min\{p,1-p\}$. Then $|x_i/\hat{p}|\ge 1$ for all $i\in V$.
So by Lemma~\ref{lem:1-disc} it follows that for at least $(8y)^{-1}2^{t}$ choices of $S\subseteq V$ we have $|D_p(S)|\ge 4^{-1}\hat{p}\sqrt{t\ln y}$.
The base case holds with $d_1=8^{-1}$ and $c_1=4^{-1}$.

Now assume $\rr>1$.
For any fixed $a\in A_\rr$ we can form a $(\rr-1,\rr-1)$-graph $H_a$ on $A_1\cup\cdots\cup A_{\rr-1}$ by letting $e\in E(H_a)$ if and only if $e\cup\{a\}\in E(H)$.
Define
\begin{align*}
Y:=\{(B_1,\dots,B_{\rr-1},a)\mid B_j\subseteq A_j, a\in A_\rr, |D_{p,H_a}(B_1,\dots, B_{\rr-1})|\ge \hat{p}c_{\rr-1}t^{(\rr-1)/2}\}
\end{align*}
By induction, for $t$ large enough, we know that for any $a\in A_{\rr}$ 
\begin{align*}
\left |\{(B_1,\dots,B_{\rr-1})\mid B_j\subseteq A_j, |D_{p,H_a}(B_1,\dots,B_{\rr-1})|\ge \hat{p}c_{\rr-1}t^{(\rr-1)/2}\}\right| \ge d_{\rr-1}2^{t(\rr-1)}/e.
\end{align*}
(Here, we have applied the statement for $(\rr-1,\rr-1)$-graphs with $y=e$.)
Let us write $d=e^{-1}d_{\rr-1}$. 
So $|Y|\ge dt2^{t(\rr-1)}$. 
This implies that out of the $2^{t(\rr-1)}$ choices of $(B_1,\dots,B_{\rr-1})$ at least $\frac12 d 2^{t(\rr-1)}$ of them satisfy that $|\{a\in A_{\rr}:|D_{p,H_a}(B_1,\dots, B_{\rr-1})|\ge \hat{p}c_{\rr-1}t^{(\rr-1)/2}\}|\ge dt/2$.
(Otherwise, $|Y| < \frac12 d 2^{t(\rr-1)}|A_\rr| + 2^{t(\rr-1)} dt/2 < dt2^{t(\rr-1)}$, a contradiction.)
Fix such a $(B_1,\dots,B_{\rr-1})$ and define for $a\in A_{\rr}$
\begin{align*}
x_a=\frac{D_{p,H_a}(B_1,\dots,B_{\rr-1})}{\hat{p}c_{\rr-1}t^{(\rr-1)/2}}.
\end{align*}
Then $|x_a|\ge1$ for at least $dt/2$ of the $a$ in $A_\rr$.
By Lemma~\ref{lem:1-disc}, we have, for $\ln y\leq dt/8$, for at least $(8y)^{-1} 2^t$ choices of $B_{\rr}\subseteq A_{\rr}$ that
\begin{eqnarray}
|D_p(B_1,\dots,B_{\rr-1},B_{\rr})|&=&\left|\sum_{a\in B_\rr}e_{H_a}(B_1, \ldots ,B_{\rr-1})-p\prod_{i=1}^\rr |B_i|\right|	\nonumber
\\
&=&\left|\sum_{a\in B_{\rr}}D_{p,H_a}(B_1,\dots,B_{\rr-1})\right| =\hat{p}c_{\rr-1}t^{(\rr-1)/2}\left|\sum_{a\in B_{\rr}}x_a\right| \nonumber
\\
&\ge& t^{\rr/2}\hat{p}c_{\rr-1}\sqrt{32^{-1}d\ln y}. \label{eq:k-1 bound}
\end{eqnarray}
As this holds for at least $\frac12 d2^{t(\rr-1)}$ choices of $(B_1,\dots,B_{\rr-1})$, it follows that~\eqref{eq:k-1 bound}
holds for at least $d(16y)^{-1}2^{t\rr}=d_ry^{-1}2^{tr}$ choices of $(B_1,\dots,B_\rr)$.
So setting, $d_r=d/16$ and $c_\rr=c_{\rr-1}\sqrt{32^{-1}d}$, the proof is finished.
\end{proof}

\begin{lemma}\label{lem:almost there}
Let $\rr\ge2$ and $p\in (0,1)$.
There exists constants $c'_\rr>0, d'_\rr>0$ such that, for each $n$ large enough and any $t$ satisfying $(\ln n)/d'_\rr\le t\le n/2$, each $\rr$-uniform hypergraph $H=(V,E)$ on $n$ vertices has pairwise disjoint subsets $B_1,\dots,B_\rr\subseteq V$ with $|B_i|\le t/\rr$ for all $i$ such that 
\begin{align*}|D_p(B_1,\dots,B_\rr)|\ge \min\{p,1-p\}c'_\rr t^{(\rr+1)/2}\sqrt{\ln(n/t)}.\end{align*}
\end{lemma}
The proof of this lemma is based on ideas from~\cite[Chapter~7]{ErSp74}.
\begin{proof}
Write $\hat{p}=\min\{p,1-p\}$ and write $\alpha=\lfloor t/\rr\rfloor^{(\rr-1)/2}\sqrt{\ln(n/t)}$.
Now partition $V$ into $\rr$ pairwise disjoint sets $A_1,\dots,A_\rr$ with $A_1,\dots,A_{\rr-1}$ each of size $\lfloor t/\rr\rfloor$ and $A_\rr$ of size $n-(\rr-1)\lfloor t/\rr\rfloor\ge n/2$.
For $a\in A_\rr$, let $H_a$ be the $(\rr-1,\rr-1)$-graph on $A_1\cup\cdots\cup A_{\rr-1}$ with $e\in E(H_a)$ if $e\cup\{a\}\in E(H)$. 

Let $c=c_{\rr-1}$ and $d=d_{\rr-1}>0$ be the constants from Lemma~\ref{lem:multi disc}. 
Setting $d'_r=\frac{d}{r+1}$, we see that we may apply Lemma~\ref{lem:multi disc} to $H_a$ with $y=n/t$ to find that when selecting ${\bf B}_i\subseteq A_i$ independently and uniformly at random for $i\in[\rr-1]$, then $|D_{p,H_a}({\bf B}_1,\dots ,{\bf B}_{\rr-1})|>\hat{p}c\alpha$ with probability at least $dt/n$ for each $a\in A_\rr$.
We may assume that $d \le 8/\rr$.
For convenience write ${\bf B}=({\bf B}_1,\dots,{\bf B}_{\rr-1})$ and let 
\begin{align*}
X({\bf B}):=|\{a\in A_\rr\mid |D_{p,H_a}({\bf B})|> \hat{p}c\alpha\}|.
\end{align*}
Then, as $|A_\rr|\ge n/2$, $\mathbb{E}(X({\bf B}))\ge dt/2$.
This implies that there exists $B=(B_1, \ldots, B_{r-1})$ with $B_i\subseteq A_i$ for $i\in[\rr-1]$ such that $X(B)\ge dt/2$.
By symmetry we may assume that $|\{a\in A_\rr\mid D_{p,H_a}(B)>\hat{p}c\alpha\}|\ge dt/4$.
Now fix $B_\rr\subseteq A_\rr$ of size $dt/4 \le t/\rr$ such that $D_{p,H_a}(B)>\hat{p}c\alpha$ for each $a\in B_\rr$.
Then 
\begin{align}
D_p(B_1,\dots,B_\rr)=\sum_{a\in B_\rr}D_{p,H_a}(B)\ge \hat{p}c'_\rr t^{(\rr+1)/2}\sqrt{\ln(n/t)}
\end{align}
for $n$ large enough, with $c'_\rr=cd/(4\rr^{(\rr-1)/2}+1)$.
This finishes the proof.
\end{proof}
We can now prove Theorem~\ref{thm:skew disc}.
\begin{proof}[Proof of Theorem~\ref{thm:skew disc}]
Let $D$ be the constant $d'_\rr$ from Lemma~\ref{lem:almost there}.
The cases $t > n/2$ follow from the case $t= n/2$.
So we may assume $(\ln n)/D \le t\le n/2$.
By the previous lemma, there is a constant $c>0$ and sets $B_1,\dots,B_\rr$ of size at most $t/\rr$ such that 
$|D_p(B_1,\dots,B_\rr)|\ge c\min\{p,1-p\}t^{(\rr+1)/2}\sqrt{\log(n/t)}$.
Now we claim that 
\begin{equation}
\sum_{S\subseteq [r]}(-1)^{|S|}D_p(\bigcup_{i\in S} B_i)=(-1)^rD_p(B_1,B_2,\dots, B_\rr),\label{eq:Dunion}
\end{equation}
which we will prove shortly. 
Let us first observe that it implies, for at least one of the $2^r-1$ nonempty subsets $S$ of $[r]$, we have 
\[|D_p(\bigcup_{i\in S}B_i)|\geq 2^{-r}D_p(B_1\ldots B_\rr)\geq  2^{-\rr}c\min\{p,1-p\}t^{(\kappa+1)/2}\sqrt{\log(n/t)}.\] 
As $|\bigcup_{i\in S}B_i)|\leq t$, setting $C=2^{-\rr}c$, this finishes the proof of the theorem.

To prove \eqref{eq:Dunion}, let us define for a subset $U=\{i_1,\ldots,i_m\}\subseteq [\rr]$ and $\alpha \in \mathbb{Z}_{\geq 0}^m$ such that $\sum_{i=1}^m \alpha_i=\rr$, $e(B_{i_1}^{\alpha_1},\ldots,B_{i_m}^{\alpha_m})$ to be the number of hyperedges of $H$ that have $\alpha_j$ endpoints in $B_{i_j}$ and define 
\[
D_p(B_{i_1}^{\alpha_1},\ldots, B_{i_m}^{\alpha_m})=e(B_{i_1}^{\alpha_1}\ldots B_{i_m}^{\alpha_m})-p\prod_{j=1}^m \binom{|B_{i_j}|}{\alpha_j}.
\]
Then for any $U=\{i_1,\ldots,i_m\}\subseteq [r]$ we have 
\begin{equation}
\label{eq:incexc}
D_p(\bigcup_{i\in U} B_i)=\sum_{\substack{\alpha\in \Z_{\geq 0}^m\\ \sum_{i=1}^m \alpha_i=r}}D_p(B_{i_1}^{\alpha_1},\ldots, B_{i_m}^{\alpha_m}).
\end{equation}
We substitute (\ref{eq:incexc}) into the left hand side of (\ref{eq:Dunion}) and examine the various contributions. 
Let us fix $U=\{i_1,\ldots,i_m\}\subseteq [\rr]$ and $\alpha\in \Z^m$ such that $\sum_{i=1}^m \alpha_i=\rr$ and such that each $\alpha_i>0$ and look at the contribution of $D_p(B_{i_1}^{\alpha_1},\ldots, B_{i_m}^{\alpha_m})$ to \eqref{eq:Dunion}.
Clearly, there is a contribution if and only if $S$ contains $U$. 
For $m'\geq m$ there are exactly $\binom{\rr-m}{m'-m}$ sets $S$ that give a contribution of $(-1)^{m'}D_p(B_{i_1}^{\alpha_1},\ldots, B_{i_m}^{\alpha_m})$. 
So the contribution of the pair $U,\alpha$ to \eqref{eq:Dunion} is given by
\[
\sum_{i=0}^{\rr-m}(-1)^{i+m}\binom{\rr-m}{i}=(-1)^m\sum_{i=0}^{\rr-m}\binom{\rr-m}{i}=\left \{\begin{array}{cc}0&\text{ if } m<\rr\\ (-1)^\rr&\text{ if } m=\rr. \end{array}\right. 
\]
This proves \eqref{eq:Dunion}.
\end{proof}

\section{Multi-colour quasi-Ramsey results for graphs}\label{sec:QR graphs}
Here we give a proof of Theorem~\ref{thm:graph multi-colour} and discuss some consequences of it.
Our proof of Theorem~\ref{thm:graph multi-colour} is based on the proof of~\cite[Theorem~2]{KLPR14+}, which in turn is inspired by a method of Erd\H{o}s and Pach~\cite{ErPa83}.
\begin{proof}[Proof of Theorem~\ref{thm:graph multi-colour}]
Let $\phi:\binom{[n]}{2}\to[\qq]$ be a colouring of the edges of the complete graph on $n \geq Ck\ln k$ vertices.
Let us write $G_j=([n],\phi^{-1}(\{j\}))$, the graph given by colour $j$ for $j\in [\qq]$.
For a set $S\subseteq V$ and $j\in [\qq]$, we define the following form of \emph{skew-discrepancy}
\begin{align*}
D_{j,\nu}(S):=D_{\rho_j,G_j}(S)-\nu |S|^{3/2}.
\end{align*}
By $D_j(S)$ we mean $D_{j,0}(S)$ and we refer to $\nu |S|^{3/2}$ as the \emph{skew factor} of the set $S$.

Let us construct a sequence of graphs as follows. 
We define $V_0:=[n]$.
For $i>0$, suppose $X_{i-1}$ and $V_{i-1}$ are given.
Then amongst all choices of $(S,j)$ where $S \subseteq V_{i-1}$ and $j$ is a colour, let $(X(i), c(i))$ maximize $D_{j,\nu}(S)$ 
and set $V_{i}:=V_{i-1}\setminus X_{i}$. 
Note that by Theorem~\ref{thm:skew disc} we alway have that $D_{c(i),\nu}(X_{i})>0$.
We stop at step $t$, the first time that $|V_{t}|< n/2$.
Define for $j\in [\qq]$, $I_j:=\{i\in[t] \mid c(i)=j\}$.
\begin{claim}\label{claim:1}
For each $j\in [\qq]$ and each $i\in I_j$, \begin{align*}\delta(G_{j}[X_i])\ge \rho_j (|X_i|-1) +\nu(|X_i|-1)^{1/2}.\end{align*}
\end{claim}
\begin{claimproof}
Suppose there exists a vertex $x\in X_i$ with strictly smaller minimum degree.
Write $n_i:=|X_i|$. We may of course assume that $n_i\ge 2$.
Set $X_i':=X_i\setminus \{x\}$. 
Then $e(X_i')=e(X_i)-\deg_{G_j}(x)>e(X_i)-\rho_j (n_i-1)-\nu(n_i-1)^{1/2}$.
So it follows that
\begin{align}
D_{j,\nu}(X_i')&>e(X_i)-\rho_j\binom{n_i-1}{2}-\nu (n_i-1)^{3/2}-\rho_j (n_i-1)-\nu(n_i-1)^{1/2}	\nonumber
\\
&=e(X_i)-\rho_j\binom{n_i}{2} -\nu((n_i-1)^{(3/2}+(n_i-1)^{1/2}).
\label{eq:deg bound}
\end{align}
Now note that $n_i^{3/2}=(n_i-1+1)n_i^{1/2}>(n_i-1)^{3/2}+(n_i-1)^{1/2}$.
This implies by~\eqref{eq:deg bound} that $D_{j,\nu}(X_i')>D_{j,\nu}(X_i)$, contradicting the maximality of $D_{j,\nu}(X_i)$.
\end{claimproof}

By Claim~\ref{claim:1} we may assume that $|X_i|\le k-1$ for all $i\in I_j$ and $j\in [\qq]$, for else we are done.
By symmetry among the colours, we may assume that 
\begin{align}
\sum_{i\in I_1} |X_{i}|\ge \frac{n}{2\qq}.	\label{eq:size bound}
\end{align}
Writing $I_1:=\{i_1,\dots,i_m\}$, we have that for each $s\in[m-\qq-1]$:
\begin{align}
\left(\sum_{j=1}^{\qq+1}|X_{i_{s+j}}|\right)^{3/2}\le (\qq+1)^{3/2}(k-1)^{3/2}.\label{eq:nu size bound}
\end{align}
We next show the following:
\begin{claim}\label{claim:exp decrease}
For each $s\in[m-\qq-1]$, $D_{1}(X_{i_{s+\qq+1}})\le \frac{5}{2(\qq+1)}D_1(X_{i_s}).$
\end{claim}
\begin{claimproof}
For any $i\ne j\in I_1$ define 
\begin{align*}
D_{1}(X_i,X_j):=e_{G_1}(X_i,X_j)-\rho_1|X_i||X_j|,
\end{align*}
where $e_{G_1}(X_i,X_j)$ denotes the number of edges between $X_i$ and $X_j$ in the graph $G_1$.
Then $D_1(X_i\cup X_j)=D_1(X_i)+D_1(X_j)+D_{1}(X_i,X_j)$.
Let $s\in[m-1]$.
Then, by maximality of $D_{1,\nu}(X_{i_{s}})$, we have $D_{1,\nu}(X_{i_{s}})\ge D_{1,\nu}(X_{i_s}\cup X_{i_{s+1}})$, which implies that
\begin{align*}
D_{1}(X_{i_{s+1}})\le -D_{1}(X_{i_s},X_{i_{s+1}})+\nu |X_{i_s}\cup X_{i_{s+1}}|^{3/2}.	
\end{align*}
Using the obvious fact that $\nu|X|^{3/2}\le \nu|Y|^{3/2}$ if $|X|\le |Y|$, this implies
\begin{align}
\sum_{t=1}^{\qq+1} tD_{1}(X_{i_{s+t}})\le -\sum_{0\le j<l\le \qq+1}D_1(X_{i_{s+j}},X_{i_{s+l}})+\binom{\qq+1}{2}\nu\left|\cup_{j=0}^{\qq+1} X_{i_{s+j}}\right|^{3/2}. \label{eq:nu bound}
\end{align}

Let us now fix $s\in[m-\qq-1]$.
Then, since $D_{1}(X)=-\sum_{c=2}^\qq D_c(X)$ for any set $X$, it follows that
\begin{align*}
-D_1(\cup_{j=0}^{\qq+1} X_{i_{s+j}})=\sum_{c=2}^\qq D_c(\cup_{j=0}^{\qq+1} X_{i_{s+j}})&= \sum_{c=2}^\qq D_{c,\nu}(\cup_{j=0}^{\qq+1} X_{i_{s+j}})+(\qq-1)\nu\left|\cup_{j=0}^{\qq+1} X_{i_{s+j}}\right|^{3/2} 
\\
&\le(\qq-1)\left(D_{1}(X_{i_s})+\nu\left|\cup_{j=0}^{\qq+1} X_{i_{s+j}}\right|^{3/2}\right),
\end{align*}
by maximality of $D_{1,\nu}(X_{i_s})$.
This clearly implies 
\begin{align}
-\sum_{j=1}^{\qq+1} D_{1}(X_{i_{s+j}})-\sum_{0\le j<l\le \qq+1}D_1(X_{i_{s+j}},X_{i_{s+l}})\le \qq D_1(X_{i_s})+(\qq-1)\nu\left|\cup_{j=0}^{\qq+1} X_{i_{s+j}}\right|^{3/2}.		\label{eq:4-union}
\end{align}
Combining~\eqref{eq:4-union} and~\eqref{eq:nu bound} we obtain
\begin{align*}
\sum_{t=2}^{\qq+1} (t-1)D_{1}(X_{i_{s+t}})\le \qq D_{1}(X_{i_s})+\left(\binom{\qq+1}{2}+\qq-1\right)\nu \left|\cup_{j=0}^{\qq+1} X_{i_{s+j}}\right|^{3/2},
\end{align*}
from which it follows, using maximality of $D_{1,\nu}(X_{i_{s+t}})$, that
\begin{align*}
\frac{\qq(\qq+1)}{2}D_{1}(X_{i_{s+\qq+1}})\le \qq D_1(X_{i_s})+(\qq(\qq+1)+\qq-1))\nu\left|\cup_{j=0}^{\qq+1} X_{i_{s+j}}\right|^{3/2}.
\end{align*}
So by~\eqref{eq:nu size bound}
\begin{align}
D_{1}(X_{i_{s+\qq+1}})\le \frac{2}{\qq+1}D_1(X_{i_s})+3(\qq+1)^{3/2}\nu(k-1)^{3/2}.
\label{eq:compare}
\end{align}
As $|V_{i}|\ge n/2$ for all $i<t$ we know by Theorem~\ref{thm:skew disc} that there exists a set $X\subseteq V_{i_s}$ of size at most $k$ whose $\rho_j$-discrepancy satisfies 
\begin{align}
D_{\rho_j,G_j}(X) \ge
\max_{j'\in [\qq]}\min\{\rho_{j'},1-\rho_{j'}\}C k^{3/2}\sqrt{\ln(C(\nu)\ln k)},	\label{eq:rho disc}
\end{align}
for some $j\in [\qq]$.
Since the skew factor of this set $X$ is at most $\nu k^{3/2}$, it follows that if $k$ (or $C(\nu)$) is large enough, then $D_1(X_{i_s})\ge   6(\qq+1)^{5/2}\nu(k-1)^{3/2}$.
Combining this with~\eqref{eq:compare} finishes the proof of the claim.
\end{claimproof}
Claim~\ref{claim:exp decrease} now implies that for $s=m-\qq-1$ we have $D_1(X_{i_s})\le \left(\frac{5}{2(\qq+1)}\right)^{s/(\qq+1)}D_1(X_{i_1})$ (where for simplicity we have assumed that $s\equiv0 \pmod{(\qq+1)}$).
Note that $D_1(X_{i_s})\geq 6(\qq+1)^{5/2}\nu(k-1)^{3/2}$ (by the proof of Claim~\ref{claim:exp decrease}) and that this is at least $1$ when $k$ is large enough.
From this we deduce that $m$ is bounded by
\begin{align*}
\frac{(\qq+1)\ln(D(X_{i_1}))}{\ln(\frac25(\qq+1))}+(\qq+1)\le \frac{2(\qq+1)(1+\ln(\frac25(\qq+1))}{\ln(\frac25(\qq+1))}\ln(k-1)=:c(\qq)\ln(k-1).
\end{align*}
So by~\eqref{eq:size bound} we deduce that at least one of the $m$ sets $X_i$ with $i\in I_1$ satisfies $|X_i|\ge \frac{C(\nu)k}{2\qq c(\qq)}$,
which for $C(\nu)$ large enough contradicts the fact that $|X_i|\le k-1$ for all $i\in I_1$. This proves the theorem.
\end{proof}

In case $\nu=0$ in Theorem~\ref{thm:graph multi-colour} the proof shows that the statement can actually be strengthened to the case that $(\rho_1,\dots,\rho_\qq)$ are not constant. Indeed, from~\eqref{eq:compare} we can directly argue that $m$ is bounded by a constant depending on $\qq$ times $k-1$. This means we do not need~\eqref{eq:rho disc}, which requires that $\rho_i$ is not too small in terms of $k$ or $C(\nu)$.
So we have the following corollary, which in particular implies that for any graph $G$ on $n$ vertices, $g(G)=\Omega(n/\ln n)$, partly answering the question of Falgas-Ravry, Markstr\"om, and Verstra\"ete~\cite{FMV15+}. 
\begin{cor}\label{cor:nu=0}
For any $\qq\ge2$ there exists a constant $C$ such that for any $k\in\N$, any $\qq$-colouring of the edges of the complete graph on $n=Ck\ln k$ vertices and any $(\rho_1,\dots,\rho_\qq)\in (0,1)^{\qq}$ such that $\sum_{i=1}^\qq\rho_i=1$, there exists a colour $j\in [\qq]$ and set of vertices $S$ of size $\ell$ at least $k$ such that the graph induced by $S$ in colour $j$ has minimum degree at least $\rho_j(\ell-1)$.
\end{cor}
\paragraph{Remark.}
By adapting some results in~\cite{KLPR14+}, which are based on discrepancy results of Spencer~\cite{Spe85} and Lov\'asz, Spencer and Vesztergombi~\cite{LSV86}, one can deduce from Theorem~\ref{thm:graph multi-colour} that there is a set $S$ of size {\em exactly} $k$ which has minimum degree at least $\rho_i(k-1)$ plus a constant times $\sqrt{(k-1)/\ln k}$.
We leave the details to the reader.

\section{A precise threshold for uniform hypergraphs}\label{sec:QR hyper}
In this section we prove Proposition~\ref{prop:linear hyper} and Theorem~\ref{thm:sharp hyper}.

\subsection{The linear regime}
We prove Proposition~\ref{prop:linear hyper} by combining a greedy deletion argument together with probabilistic thinning, similar to what was done for graphs in~\cite{KPPR15}.
We require the following concentration inequality~\cite[Corollary~6.10]{McD89}.

\begin{theorem}[McDiarmid~\cite{McD89}]\label{thm:colin}
Let $Z_1, \ldots, Z_n$ be random variables with $Z_i$ taking values in a set $A_i$ and let $Z = (Z_1, \ldots, Z_n)$. Let $f: \prod A_i \rightarrow \mathbb{R}$ be measurable. Suppose there exist constants $c_1, \ldots, c_n$ such that for each $k= 1, \ldots, n$ 
\begin{align*}
\left|
 	\mathbb{E}( f(Z) \mid Z_i=z_i, i\in[k-1], Z_k=z_k) 
- 	\mathbb{E}(f(Z) \mid Z_i=z_i, i\in[k-1], Z_k=z_k')
 \right| \leq c_k
\end{align*}
for all $(z_1, \ldots, z_{k-1}) \in \prod_{i=1}^{k-1}A_i$ and $z_k, z_k' \in A_k$. Then for all $t>0$ we have
\[
\mathbb{P}( |f(Z) - \mathbb{E}(f(Z))| > t ) \leq \exp\left(-2t^2 \left/ \sum_{i=1}^n c_i^2\right.\right).
\]
\end{theorem}

Using this result, we can prove the following lemma, which is a standard application of martingale inequalities, but we spell out the details for completeness.

\begin{lemma}
Let $H = (V,E)$ be an $r$-uniform hypergraph with $N$ vertices and $p \binom{N}{r}$ edges. If $S \subseteq V$ is a uniformly random subset of $n$ distinct vertices, then for any $p>\eps>0$,
\[
\mathbb{P} \left( e(H[S]) \leq (p- \varepsilon)\binom{n}{r} \right) < \exp\left( \frac{-2\varepsilon^2(n- 2(r-1))}{r^2} \right).
\]
\end{lemma}
\begin{proof}
We formulate the setup to apply Theorem~\ref{thm:colin}. 
Pick the random subset $S$ by picking its vertices one at a time uniformly at random from the pool of remaining vertices, and let $Z_1, \ldots, Z_n$ be the vertices picked, and let $Z = (Z_1, \ldots, Z_n)$.

For $\mathbf{v} = (v_1, \ldots, v_n) \in V^n$, write $H[\mathbf{v}]:=H[\{v_1, \ldots v_n\}]$.
Let $f: V^n \rightarrow \mathbb{N}$ be defined by setting $f(v_1, \ldots, v_n)$ to be the number of edges in $H[\mathbf{v}]$. 
Note that 
\[
\mathbb{P} \left( e(H[S]) \leq (p- \varepsilon)\binom{n}{r} \right)
= \mathbb{P} \left( f(Z) \leq (p- \varepsilon)\binom{n}{r} \right).
\]
Write $V^{(k)}$ for the set of $k$-component vectors in which all components are distinct. Furthermore, given $\mathbf{z} = (z_1, \ldots, z_k) \in V^k$, 
write $V^{(n)}|\mathbf{z}$ for the set of vectors in $V^{(n)}$ whose first $k$ components are $(z_1, \ldots, z_k)$.

Given two vectors $\mathbf{z} = (z_1, \ldots, z_{i-1}, z_i)$ and $\mathbf{z'} = (z_1, \ldots, z_{i-1}, z'_i)$ $\in V^{(i)}$, we define a function $g:V^{(n)}|\mathbf{z} \rightarrow V^{(n)}|\mathbf{z'}$ such that $g$ fixes $\mathbf{v}$ if $z_i'$ occurs as a component of $\mathbf{v}$ and replaces $z_i$ with $z'_i$ in $\mathbf{v}$ if $z_i'$ does not occur as a component in $\mathbf{v}$. It is easy to see that $g$ is a bijection.

Now we check the bounded difference condition in Theorem~\ref{thm:colin}. 
Note first that for  $\mathbf{z} = (z_1, \ldots, z_{i-1},z_i)\in V^{(i)}$,
\[
\mathbb{E}\left( f(Z) \mid (Z_1,\ldots,Z_i)=\mathbf{z}\right)
=
\sum_{\mathbf{v} \in V^{(n)}|\mathbf{z}} \binom{N-i}{n-i}^{-1}e(H[\mathbf{v}]).
\]
Taking $\mathbf{z'} = (z_1, \ldots, z_{i-1}, z_i') \in V^{(i)}$, we have
\begin{align*}
&\left| \mathbb{E}\left( f(Z) \mid (Z_1,\ldots,Z_i)=\mathbf{z}\right) 
- 	\mathbb{E}\left(f(Z) \mid (Z_1,\ldots,Z_i)=\mathbf{z'}\right) \right|\\
& = \bigg|\sum_{\mathbf{v} \in V^{(n)}|\mathbf{z}} \binom{N-i}{n-i}^{-1}e(H[\mathbf{v}]) 
-  \sum_{\mathbf{v} \in V^{(n)}|\mathbf{z'}} \binom{N-i}{n-i}^{-1}e(H[\mathbf{v}]) \bigg| 
\\
& = \bigg|\sum_{\mathbf{v} \in V^{(n)}|\mathbf{z}} \binom{N-i}{n-i}^{-1}(e(H[\mathbf{v}]) -  e(H[g(\mathbf{v})]) \bigg|
\\ 
& \leq \max_{\mathbf{v} \in V^{(n)}|\mathbf{z}} |e(H[\mathbf{v}]) -  e(H[g(\mathbf{v})])| \leq \binom{n}{r-1}.
\end{align*}
The last quantity is bounded above by $\binom{n}{r-1}$ because $H[\mathbf{v}]$ and $H[g(\mathbf{v})]$ are two hypergraphs that differ in at most one vertex. Now observing that $\mathbb{E}(f(Z))=p\binom{n}{r}$ and applying Theorem~\ref{thm:colin} with $c_k = \binom{n}{r-1}$ for all $k$ yields the result.
\end{proof}

We can now give a proof of Proposition~\ref{prop:linear hyper}.

\begin{proof}[Proof of Proposition~\ref{prop:linear hyper}]
Assume $\sum_{i=1}^\qq\rho_i<1 - \varepsilon$ for some $\varepsilon >0$ and let $N=Ck$, where $C$ is to be determined later. Given any $\qq$-colouring of the edges of the complete $r$-uniform hypergraph on $N$ vertices, let $H_i$ be the subhypergraph consisting of edges coloured $i$ and let $p_i$ be the edge density of $H_i$. Then for some $i$, we must have that $p_i > \rho_i +\eps/q$.  Set $\varepsilon'=\eps/q$. 
We may assume without loss of generality that $p_1 >\rho_1 + \varepsilon'$.  

Now, starting with $H_1$ and $n=N$, we repeatedly remove an arbitrary vertex of degree less than $[\rho_1 + (\varepsilon' /2)]\binom{n-1}{r-1}$. If we continue for $t$ iterations, then we have removed at most
\[
\sum_{i=1}^{t}   [\rho_1 + (\varepsilon' /2)]\binom{N-i}{r-1}
 = [\rho_1 + (\varepsilon' /2)] \left( \binom{N}{r} - \binom{N-t}{r} \right)
\] 
vertices from $H_1$.
So after $t$ iterations the number of vertices is $n=N-t$ and the number of edges remaining in the hypergraph is at least
\[
(\varepsilon' /2) \binom{N}{r}.
\]
Since this hypergraph can only have at most $\binom{n}{r}$ edges, it follows that 
\[
(\varepsilon' /2) \binom{N}{r} \leq \binom{N-t}{r}.
\]
It is easy to see that there exists $c = c(\varepsilon', r)$ such that for $n = N-t \leq  cN$, this inequality fails. Hence we find a set $T$ of vertices such that with $|T| = n = cN$ vertices such that every vertex in $H_1[T]$ has degree at least $[\rho_1 + (\varepsilon' /2)]\binom{n-1}{r-1}$ in $H_1[T]$. We want that $|T|\geq k$, so it suffices to take $C\geq 1/c$.

Finally we pick $S \subseteq T$ uniformly at random such that $|S|=k$. 
Let $v \in T$. Then, conditional on $v\in S$, the set $S\setminus\{v\}$ is a uniformly random set $S' \subseteq T\setminus\{v\}$ such that $|S'|=k-1$.
Now, if $H'$ denotes the $(r-1)$-uniform hypergraph $H'$ on $T\setminus\{v\}$ induced by the at least $[\rho_1 + (\varepsilon' /2)]\binom{n-1}{r-1}$ edges incident with $v$, then the degree of $v$ is the same as $e(H'[S'])$. So it follows from the previous lemma that, conditional on $v\in S$, the probability that the degree of  $v$ is at most $ [\rho_1 + (\varepsilon' /4)]\binom{k-1}{r-1}$ is exponentially small in $k$. Since the probability that $v\in S$ is  $k/n$, we have that, unconditionally, the probability that there exists $v\in T$ with degree at most $[\rho_1 + (\varepsilon' /4)]\binom{k-1}{r-1}$ is at most $n\cdot\frac{k}{n}\exp(-\Omega(k)) \to 0$ as $k\to\infty$. We conclude that for large enough $k$ there exists $S$ of size $k$ such that $H_1[S]$ has minimum degree at least $[\rho_1 + (\varepsilon' /4)]\binom{k-1}{r-1}$, as required.
\end{proof}

\subsection{From polynomial to super-polynomial growth}
Although we treat the significantly more general situation of hypergraphs and multiple biased colours,
our proof of Theorem~\ref{thm:sharp hyper} has strong similarities to that of~\cite[Theorem~3]{KPPR15}.
\begin{proof}[Proof of Theorem~\ref{thm:sharp hyper}\ref{itm:sharp hyper,upper}]
Let $c$ be the constant from Theorem~\ref{thm:skew disc} and let $C:=\max_{j\in [\qq]}(c\rho_j)^{-1}$.
Define for $\nu\ge 0$ and $j\in [\qq]$ the following form of skew discrepancy for a set $S\subseteq V$:
\begin{align*}
D_{\nu,j}(S):=D_{\rho_j,H_j}(S)-\nu |S|^{(\rr+1)/2}\sqrt{\ln |S|}.
\end{align*}
Let $X\subseteq V$ attain the maximum skew discrepancy over all subsets of $V$ and $j\in [n]$.
By symmetry we may assume that it is attained at colour $1$.
Using that $\binom{\ell-1}{\rr}+\binom{\ell-1}{\rr-1}=\binom{\ell}{\rr}$, we find by a similar argument as in the proof of Claim~\ref{claim:1}, that 
\begin{align*}
\delta(H_1[X])\ge \rho_1\binom{|X|-1}{\rr-1}+|X|^{(\rr-1)/2}\nu \sqrt{\ln|X|}.
\end{align*}
So it now suffices to show that $|X|\ge k$.
By Theorem~\ref{thm:skew disc} there exists a set $Y\subseteq V$ of size at most $k^{2\rr/(\rr+1)}$ such that $D_{\rho_1,H_1}(Y)\ge c\min\{\rho_1,1-\rho_1\}k^{\rr}\nu C\sqrt{(1+\varepsilon)\ln k}\ge k^{\rr}\nu \sqrt{(1+\varepsilon) \ln k}$. 
As $\varepsilon>0$, the skew factor of $X$ is dominated by $\nu k^{\rr}\sqrt{(1+\varepsilon)\ln k}$ and hence for $k$ large enough we know that $D_{\nu,j}(X)\ge k^{\rr}$. 
This clearly implies that $|X|\ge k$ and finishes the proof.
\end{proof}

For the proof of Theorem~\ref{thm:sharp hyper}\ref{itm:sharp hyper,lower}, first we describe the expected behaviour of what we refer to here as {\em $t$-dense sets} --- vertex subsets that induce average degree $\overline{\deg}$ at least $t$ --- in the random $\rr$-uniform hypergraph $\Hyp{\rr}{n}{\rho}$ with vertex set $[n] =\{1,\dots,n\}$ and hyperedge probability $\rho$.
For this, we need a result best stated with large deviations notation, cf.~\cite{DeZe98}.  For $\rho\in(0,1)$, let
\begin{align*}
\Lambda^*_\rho(x) = 
\begin{cases}
x \ln \frac{x}{\rho} + (1 - x) \ln \frac{1 - x}{1-\rho} & \mbox{for $x\in [0, 1]$}\\
\infty & \mbox{otherwise}
\end{cases}
\end{align*}
(where $\Lambda^*_\rho(0) =-\ln(1-\rho) $ and $\Lambda^*_\rho(1) = -\ln\rho$).
This is the Fenchel-Legendre transform of the logarithmic moment generating function associated with the Bernoulli distribution with probability $\rho$ (cf.~Exercise 2.2.23(b) of~\cite{DeZe98}).  Some calculus checks that $\Lambda^*_\rho(x)$ has a global minimum of $0$ at $x = \rho$, is strictly decreasing on $[0, \rho)$ and strictly increasing on $(\rho,1]$.  The following is a straightforward adaptation of Lemma~2.2(i) in~\cite{KaMc10} and bounds the probability that a given subset of $k$ vertices in $\Hyp{\rr}{n}{\rho}$ is $t$-dense.

\begin{lemma}
\label{lem:A_n}
Given $t,k$ with $t \ge \rho\binom{k - 1}{\rr-1}$,
\begin{align*}
\Pr\left(\overline{\deg}\left(\Hyp{\rr}{k}{\rho}\right) \ge t\right) \le \exp\left(-\binom{k}{\rr} \Lambda^*_\rho\left( t\left/\binom{k - 1}{\rr-1}\right. \right) \right).
\end{align*}
\end{lemma}

\begin{proof}[Proof of Theorem~\ref{thm:sharp hyper}\ref{itm:sharp hyper,lower}]
For any $\eta>1$ let
\begin{align*}
n = \left\lfloor \frac{k^{\nu(k)^2+1}}{\eta e} \right\rfloor,
\end{align*}
where $k$ is some large enough integer. For each $i\in[\qq]$ we write
\begin{align*}
f_i(\ell) = \rho_i\binom{\ell-1}{\rr-1}+\nu(\ell) \sqrt{\rr\binom{\ell-1}{\rr-1}\ln \ell}.
\end{align*}
Let us consider a random $\qq$-colouring of $\binom{[n]}{\rr}$, the hyperedges of $K^{(\rr)}_n$, where independently and uniformly each hyperedge is assigned the colour $i$ with probability $\rho_i$. So, writing $H_i$ for the subhypergraph induced on $[n]$ by the hyperedges of colour $i$, we see that $H_i$ is distributed as the random $\rr$-uniform hypergraph $\Hyp{\rr}{n}{\rho_i}$.

Given a subset $S \subseteq [n]$ of $\ell\ge k$ vertices, let $A_S$ be the event that $\overline{\deg}(H_i[S]) \ge f_i(\ell)$ for some $i\in [\rr]$.
Since $f_i(\ell)\ge \rho_i\binom{\ell-1}{\rr-1}$ for each $i$, we have by Lemma~\ref{lem:A_n} that
\begin{align*}
\Pr(A_S)
& \le \sum_{i=1}^\qq\exp\left(-\binom{\ell}{\rr}\Lambda^*_{\rho_i}\left(f_i(\ell)\left/\binom{\ell-1}{\rr-1}\right.\right)\right) \\
& \le \sum_{i=1}^\qq\exp\left(-\binom{\ell}{\rr}\Lambda^*_{\rho_i}\left(\rho_i+\nu(\ell)\sqrt{\rr\ln \ell\left/\binom{\ell-1}{\rr-1}\right.}\right)\right).
\end{align*}
Now, writing
\begin{align*}
\eps=\eps(\ell)=\nu(\ell)\sqrt{\rr\ln \ell\left/\binom{\ell-1}{\rr-1}\right.},
\end{align*}
we have by Taylor expansion of $\Lambda^*_{\rho_i}$ (assuming $\eps < \min\{\rho_i,1-\rho_i\}$) that
\begin{align*}
\Lambda^*_{\rho_i}(\rho_i+\eps) 
& = (\rho_i+\eps) \ln \left(1+\frac{\eps}{\rho_i}\right) + (1 -\rho_i-\eps) \ln \left(1-\frac{\eps}{1-\rho_i}\right)\\
& = \sum_{j=1}^\infty\frac{\eps^{2j}}{(2j-1)2j}\left(\frac{1}{{\rho_i}^{2j-1}}+\frac{1}{(1-\rho_i)^{2j-1}}\right) + \sum_{j=1}^\infty\frac{\eps^{2j+1}}{2j(2j+1)}\left(\frac{1}{(1-\rho_i)^{2j}}-\frac{1}{{\rho_i}^{2j}}\right)\\
& = \frac{\eps^2}{2\rho_i(1-\rho_i)} + O(\eps^3) \ge \eps^2
\end{align*}
for $\eps$ small enough (and hence $k$ large enough).
So the probability that $A_S$ holds for some set $S \subseteq [n]$ of $\ell\ge k$ vertices is at most 
\begin{align*}
\sum_{S \subseteq [n], |S| \ge k} \Pr(A_S)
& \le \sum_{\ell = k}^n \binom{n}{\ell} \sum_{i=1}^\qq\exp\left(-\binom{\ell}{\rr}\Lambda^*_{\rho_i}(\rho_i+\eps)\right)\\ 
& \le \qq\sum_{\ell = k}^n \left(\frac{en}\ell \exp\left(-\frac{1}{\rr}\binom{\ell-1}{\rr-1}\eps^2\right) \right)^{\ell}
 \le \qq\sum_{\ell=k}^n \eta^{-\ell}< 1,
\end{align*}
where in this sequence of inequalities we have used the definition of $n$, the fact that $\ell\ge k$ and $\eta>1$, and a choice of $k$ large enough.
Thus for $k$ large enough there is a $\qq$-colouring of the edges of $K^{(\rr)}_n$ where for each $i\in[\qq]$ every vertex subset of size $\ell\ge k$ induces a subhypergraph in colour $i$ with average degree less than $f_i(\ell)$, 
 so the result follows.
\end{proof}

\section{Concluding remarks and open questions}\label{sec:remarks}
Let us introduce some notation to facilitate our discussion.
Fix $\qq \ge 2$ and let $(\rho_i)_{i=1}^\qq$ be a sequence of $\qq$ numbers in $[0,1]$. 
Given a colouring $\phi$ of the complete $\rr$-uniform hypergraph $K^{(\rr)}_n$ on vertex set $[n]$ that assigns each hyperedge a colour from $[\qq]$, we let $H_{\phi,j}$ denote the subhypergraph $([n],\phi^{-1}(j))$ induced by all hyperedges of colour $j$ for $j\in[\qq]$.
The basic question now becomes the following:
for any $k$, what is the smallest number $n:=R^{(\rr)}_{(\rho_i)_i}(k)$ such that, for any $\qq$-colouring $\phi$                                           
of the hyperedges of $K^{(\rr)}_n$, there is guaranteed to be a subset $S\subseteq [n]$ of size $\ell$ at least $k$ such that the subhypergraph $H_{\phi,j}[S]$ induced on $S$ in colour $j$ has minimum degree at least $\rho_j\binom{\ell-1}{\rr-1}$ for some $j\in[\qq]$?
We may also ask this question with average degree instead of minimum degree and denote the corresponding number $\overline{R}^{(\rr)}_{(\rho_i)_i}(k)$.
Clearly $\overline{R}^{(\rr)}_{(\rho_i)_i}(k) \le R^{(\rr)}_{(\rho_i)_i}(k)$ always.
We refer to $R^{(\rr)}_{(\rho_i)_i}(k)$ and $\overline{R}^{(\rr)}_{(\rho_i)_i}(k)$ as \emph{$\qq$-colour hypergraph quasi-Ramsey numbers}. Note that when $\sum_{i=1}^q\rho_i=q$ we retrieve the ordinary hypergraph Ramsey-numbers.

With this notation we see that for $\sum_{i=1}^\qq\rho_i<1$, Proposition~\ref{prop:linear hyper} shows that $R^{(\rr)}_{(\rho_i)_i}(k)$, and hence $\overline{R}^{(\rr)}_{(\rho_i)_i}(k)$, has linear growth in $k$; Theorem~\ref{thm:sharp hyper} precisely describes the transition from polynomial to super-polynomial growth of the $\qq$-colour hypergraph quasi-Ramsey numbers. In particular, for $\sum_{i=1}^\qq\rho_i>1$, Theorem~\ref{thm:sharp hyper}\ref{itm:sharp hyper,lower} implies that $\overline{R}^{(\rr)}_{(\rho_i)_i}(k)$, and hence $R^{(\rr)}_{(\rho_i)_i}(k)$, is at least singly exponential in $k$. 
This implies for hypergraph quasi-Ramsey numbers that, irrespective of a well-known conjecture of Erd\H{o}s, Hajnal and Rado~\cite{EHR65} (concerning the case $\sum_{i=1}^\qq\rho_i=\qq$), there must be a transition for $\rr$-uniform hypergraphs with $\rr\ge4$ from singly exponential to doubly exponential (or higher) growth in $k$ that takes place for $1<\sum_{i=1}^\qq\rho_i\leq q$. 
It would be an interesting challenge to understand the nature of this transition.

We note that if all the $\rho_i$ are uniformly bounded below $1$, Conlon, Fox and Sudakov~\cite{CFS10,CFS11} proved results that imply $R^{(\rr)}_{(\rho_i)_i}(k)$, and hence $\overline{R}^{(\rr)}_{(\rho_i)_i}(k)$, has growth that is at most singly exponential in $k$:
\begin{proposition}[Conlon, Fox and Sudakov~\cite{CFS10,CFS11}]\label{prop:CFS}
Let $\rr\ge2$. Fix $\qq \ge 2$ and $\eps>0$ and let $\rho_1=\dots=\rho_\qq = 1-\eps$. Then
\begin{align*}
R^{(\rr)}_{(\rho_i)_i}(k) =
\begin{cases} 
2^{O(k^2)} &\text{if $\qq = 3$~\cite[Theorem~2]{CFS11} and}  \\ 
2^{O(k^D)} &\text{if $\qq \ge4$~\cite[Proposition~6.3]{CFS10}}
\end{cases}
\end{align*}
where $D>0$ is a fixed constant that depends on $\rr$, $\qq$ and $\eps$.
\end{proposition}

Along these lines, a first question to resolve is perhaps whether a strengthening of Proposition~\ref{prop:CFS} holds: given $\rr\ge3$, $\qq\ge2$ and $\eps>0$, is there some $D$ such that $\ln R^{(\rr)}_{(\rho_i)_i}(k) = O(k^D)$ if $\sum_{i=1}^\qq\rho_i < \qq-\eps$?
Or could it instead be the case, say, that, given $\rr\ge4$, $\qq\ge2$ and $\eps>0$, there is some $D>0$ such that $\ln\ln \overline{R}^{(\rr)}_{(\rho_i)_i}(k) =\Omega(k^D)$ if $\rho_1=1$ and $\rho_i =\eps/(\qq-1)$ for $i\in\{2,\dots,\qq\}$?
These questions can be considered part of a refinement of a problem of Erd\H{o}s (cf.~\cite[pp.~21--22]{NeRo90}), a problem he described as ``interesting and mysterious'' and for whose solution he offered \$500. 
Borrowing his intuition, it might be more natural to believe that the answer to the first question is `yes' and to the second `no'.
On the other hand, in light of a result of Erd\H{o}s and Hajnal that was mentioned in the introduction, the answers could depend on $\rr$ and $\qq$.

\subsection*{Acknowledgement}
We would like to thank Maria Axenovich for suggesting the study of multi-colour quasi-Ramsey numbers.
We are moreover grateful to an anonymous referee for valuable comments, significantly improving the presentation of the paper.

\bibliographystyle{abbrv}
\bibliography{multiqr}

\end{document}